\documentclass[12pt,fleqn,draft]{article} 
\usepackage{amsmath,amssymb,amsthm,amsfonts,bm}
\usepackage{enumerate}
\usepackage{indentfirst}
\usepackage{cite}
\usepackage[usenames,dvipsnames]{color}

\def\luca #1{{\color{black}#1}}

\makeatletter
\def\@cite#1#2{[{{\bfseries #1}\if@tempswa , #2\fi}]}
\renewcommand{\section}{%
\@startsection{section}{1}{\z@}
{0.5truecm plus -1ex minus -.2ex}%
{1.0ex plus .2ex}{\bfseries\large}}
\def\@seccntformat#1{\csname the#1\endcsname.\ }
\makeatother

\numberwithin{equation}{section} 
\newtheorem{thm}{Theorem}[section]

\newtheorem{lem}[thm]{Lemma}

\theoremstyle{definition}
\newtheorem{df}{Definition}[section]
\newtheorem{remark}{Remark}[section]

\newtheorem*{prth1.1}{Proof of Theorem 1.1}
\newtheorem*{prth1.2}{Proof of Theorem 1.2}
\newtheorem*{prth1.3}{Proof of Theorem 1.3}

\newcommand{\ep}{\varepsilon}

\let\widehat\widehat

\def\Pi{\widehat\pi}

\begin{document}

\begin{center}
    \Large{{\bf 
Nonlocal to local convergence of phase field systems 
\\ with inertial term}}
\end{center}
\vspace{5pt}
\begin{center}
Pierluigi Colli\\
    \vspace{2pt}
    Dipartimento di Matematica ``F. Casorati'', 
    Universit\`a di Pavia\\
    and Research Associate at the IMATI -- C.N.R. Pavia\\ 
    via Ferrata 5, 27100 Pavia, Italy\\
    {\tt pierluigi.colli@unipv.it}\\
    \vspace{14pt}
Shunsuke Kurima\,\footnote{Corresponding author}\\
    \vspace{2pt}
    Department of Mathematics, 
    Tokyo University of Science\\
    1-3, Kagurazaka, Shinjuku-ku, Tokyo 162-8601, Japan\\
    {\tt shunsuke.kurima@gmail.com}\\
    \vspace{14pt}
Luca Scarpa\\
    \vspace{2pt}
    Department of Mathematics, 
    Politecnico di Milano\\
    Via E.~Bonardi 9, 20133 Milano, Italy\\
    {\tt luca.scarpa@polimi.it}\\
    \vspace{14pt}
\end{center}


\vspace{2pt}
\newenvironment{summary}
{\vspace{.5\baselineskip}\begin{list}{}{%
     \setlength{\baselineskip}{0.85\baselineskip}
     \setlength{\topsep}{0pt}
     \setlength{\leftmargin}{12mm}
     \setlength{\rightmargin}{12mm}
     \setlength{\listparindent}{0mm}
     \setlength{\itemindent}{\listparindent}
     \setlength{\parsep}{0pt}
     \item\relax}}{\end{list}\vspace{.5\baselineskip}}
\begin{summary}
{\footnotesize {\bf Abstract.} 
This paper deals with a nonlocal model for a hyperbolic phase field system coupling the standard energy balance equation 
for temperature with a dynamic for the phase variable:
\luca{the latter includes an inertial term and 
a nonlocal convolution-type operator where the family of kernels 
depends on a small parameter.} We rigorously study the asymptotic convergence of the system as the approximating parameter tends to zero and 
we \luca{obtain} at the limit the local system with the elliptic laplacian operator acting on the phase variable. Our analysis is based on some asymptotic properties 
\luca{on nonlocal-to-local convergence} 
that have been recently and successfully applied to families of Cahn--Hilliard models.}
\end{summary}
\vspace{10pt}

\newpage

\section{Introduction}\label{Sec1}
In this paper we consider the local problem
%
%
 \begin{equation*}\tag*{(P)}\label{P}
     \begin{cases}
         \theta_{t} + \varphi_{t} - \Delta\theta = f   
         & \mbox{in}\ \Omega\times(0, T), 
 \\[1.5mm]
         \varphi_{tt} + \varphi_{t} - \Delta\varphi
           + \beta(\varphi) + \pi(\varphi) 
         = \theta   
         & \mbox{in}\ \Omega\times(0, T), 
 \\[1.5mm]
         \partial_{\nu}\theta = \partial_{\nu}\varphi = 0                                  
         & \mbox{on}\ \Gamma\times(0, T),
 \\[1.5mm]
        \theta(0) = \theta_{0},\ 
        \varphi(0) = \varphi_{0},\ 
        \varphi_{t}(0) = v_{0}
        & \mbox{in}\ \Omega,   
     \end{cases}
\end{equation*}
and the corresponding family of nonlocal problems
%
%
 \begin{equation*}\tag*{(P)$_{\ep}$}\label{Pep}
     \begin{cases}
         (\theta_{\ep})_{t} + (\varphi_{\ep})_{t} - \Delta\theta_{\ep} = f   
         & \mbox{in}\ \Omega\times(0, T), 
 \\[1.5mm]
         (\varphi_{\ep})_{tt} + (\varphi_{\ep})_{t} 
           + a_{\ep}\varphi_{\ep} - J_{\ep}\ast\varphi_{\ep} 
           + \beta(\varphi_{\ep}) + \pi(\varphi_{\ep}) 
         = \theta_{\ep}   
         & \mbox{in}\ \Omega\times(0, T), 
 \\[1.5mm]
         \partial_{\nu}\theta_{\ep} = 0                                  
         & \mbox{on}\ \Gamma\times(0, T),
 \\[1.5mm]
        \theta_{\ep}(0) = \theta_{0,\ep},\ 
        \varphi_{\ep}(0) = \varphi_{0, \ep},\ 
        (\varphi_{\ep})_{t}(0) = v_{0,\ep}
        & \mbox{in}\ \Omega,
     \end{cases}
\end{equation*}
where $\Omega$ is a bounded domain in $\mathbb{R}^d$ ($d = 2, 3$)
with smooth boundary $\Gamma:=\partial\Omega$, and $T>0$ is a fixed final time.
Here, $\ep \in (0, 1)$ is a given parameter entering the nonlocal problems
through the family $\{J_\ep\}_\ep$ of convolution kernels: for every $\ep\in(0,1)$
the nonlocal terms in \ref{Pep} are defined as $(J_{\ep}\ast\varphi_{\ep})(\cdot) 
    := \int_{\Omega} J_{\ep}(\cdot - y)\varphi_{\ep}(y)\,dy$ and
$a_{\ep}(\cdot) := \int_{\Omega} J_{\ep}(\cdot - y)\,dy$.
\luca{The nonlinearity of the problems \ref{P} and \ref{Pep}
is represented by the sum $\beta+\pi$, where $\beta$ is a maximal monotone graph
in $\mathbb R\times\mathbb R$ and $\pi$ is a Lipschitz continuous function.
Hence, $\beta+\pi$ can be interpreted as the subdifferential in the sense of convex analysis 
of a typical double-well potential.
}

\medskip

The main result of the work is the rigorous proof that, under a precise scaling of
the family of convolution kernels $\{J_\ep\}_\ep$ (see assumption {\bf A1} below),
the solutions of the nonlocal problems \ref{Pep} converge in suitable topologies 
to a weak solution of the local problem \ref{P}. This is contained in Theorem \ref{maintheorem2} below.
As a by-product, this procedure 
also shows existence of weak solutions of the local problem.

\medskip

The class of problems to whom \ref{P} and \ref{Pep} belong is usually ascribed to hyperbolic phase field models,
with hyperbolicity in the phase equation, or phase field systems with inertial term. They have been considered and analyzed in the scientific literature in the last decades: the reader may see, e.g., \cite{cgi, GMPZ, jiang, kur1,
K7, kur2, mirqui, rocschi} to name a few contributions. In particular, a derivation of the local model \ref{P} from the physical laws is performed in \cite{cgi, mirqui}, while the nonlocal model has been recently considered in  \cite{kur1, K7, kur2}. Our aim in this paper is to \luca{relate} the nonlocal model to the local one by a rigorous asymptotic analysis.

\medskip

The literature on nonlocal Cahn--Hilliard type systems has witnessed an
extensive development in recent times:
with no sake of completeness, we mention the works  
\cite{ab-bos-grass-NLCH, gal-gior-grass-NLCH, gal-grass-NLCH} and the references therein,
as well as \cite{gal-gior-gras-SEP,poiatti-SEP} dealing with strict separation property and \cite{frig-gal-gras-DEG} for the case of degenerate mobility.
In the last years an important
stream of research has been addressed to the asymptotic convergence 
of nonlocal Cahn--Hilliard systems towards their local counterparts.
The first results in this direction were established in \cite{MRT18, DRST,DST2021},
where nonlocal-to-local convergence of the some classes of Cahn--Hilliard equations
was established both in the periodic and in the viscous homogeneous Neumann setting,
in the case of non-$W^{1,1}$ kernels. Subsequently, the  analogous results in the case of
$W^{1,1}$ kernels was the subject of \cite{DST_JDE2021}. These contributions were
based both on Gamma-convergence results established by Ponce in \cite{ponce}
and ideas from the seminal works of Sandier and Serfaty in \cite{sand-serf}.
More recently, the study of nonlocal-to-local convergence of phase-field models 
has been expanded and thoroughly developed: we mention, for example, 
the works~\luca{\cite{AT, DRST2023, kur3}} on coupled systems, \cite{ca-el-sk-DEG,el-sk-DEG} on degenerate mobilities, and \cite{AH} on rates of convergence.

\medskip

The main result of this work is a natural development of the stream of research 
on nonlocal-to-local convergence of phase-field models that has been widely carried out 
in the last years. Here, we provide the very first extension of such asymptotic results 
also to the setting of inertial terms and second-order-in-time Allen--Cahn type equations.

\medskip
    
The paper is organized as follows. 
Section~\ref{sec:main} specifies the setting of the work, 
the assumptions on  the data, the concept  of solutions  for both problems, 
and the main nonlocal-to-local convergence result
in Theorem \ref{maintheorem2} below.
In Section~\ref{Sec:prelim} we collect some useful technical preliminiaries 
on nonlocal-to-local convergence of functionals, as well as 
some generalised nonlocal compactness results that will be used in the paper.
Eventually, in Section~\ref{sec:proof} we present the proof of the main result:
this is structured in two parts, consisting of proving uniform estimates on 
the nonlocal  problems independently of $\ep$ and then passing to  the limit as
$\ep\searrow0$.

\section{Setting and main results}
\label{sec:main}
We specify here the variational setting of the work, the 
assumptions, and the main results that we prove.
Throughout the paper, $\Omega\subset\mathbb R^d$, $d\in\{1,2,3\}$,
is a smooth bounded domain and $T>0$ is  a fixed final time.

\medskip

We define the Hilbert spaces 
   $$
   H:=L^2(\Omega), \quad V:=H^1(\Omega), \quad 
   W:=\bigl\{z\in H^2(\Omega)\ |\ \partial_{\nu}z = 0 \quad 
   \mbox{a.e.\ on}\ \partial\Omega\bigr\}
   $$
 with inner products 
 \begin{align*} 
 (u_{1}, u_{2})_{H}&:=\int_{\Omega}u_{1}u_{2}\,dx \quad \mbox{for}\ u_{1}, u_{2} \in H, \\
 (v_{1}, v_{2})_{V}&:=
 \int_{\Omega}\nabla v_{1}\cdot\nabla v_{2}\,dx + \int_{\Omega} v_{1}v_{2}\,dx 
 \quad \mbox{for}\ v_{1}, v_{2} \in V,\\
 (w_{1}, w_{2})_{W}&:=
 \int_{\Omega}\Delta w_{1}\Delta w_{2}\,dx + \int_{\Omega} w_{1}w_{2}\,dx 
 \quad \mbox{for}\ w_{1}, w_{2} \in W,
\end{align*}
 respectively, and with the related Hilbertian norms.
 We denote by $V^{*}$ the dual space of $V$ and by
 $\langle\cdot, \cdot\rangle_{V^*, V}$ duality pairing between $V^*$ and $V$. 
 Moreover, in this paper, the bijective mapping $F : V \to V^{*}$ and 
 the inner product in $V^{*}$ are defined as 
    \begin{alignat*}{2}
    \langle Fv_{1}, v_{2} \rangle_{V^*, V} &:= 
    (v_{1}, v_{2})_{V} \quad&&\mbox{for}\ v_{1}, v_{2}\in V,   
    \\[1mm]
    (v_{1}^{*}, v_{2}^{*})_{V^{*}} &:= 
    \left\langle v_{1}^{*}, F^{-1}v_{2}^{*} 
    \right\rangle_{V^*, V} 
    \quad&&\mbox{for}\ v_{1}^{*}, v_{2}^{*}\in V^{*}. 
    \end{alignat*}  
 We identify as usual the space $H$ with its dual $H^*$, whence we have that
 \[
 W\hookrightarrow V \hookrightarrow H \hookrightarrow V^*.
 \]
with dense, continuous, and compact inclusions. Please note that we will use the same notation $\| \cdot\|_X $ for the norm in a normed space $X$ and for any power of it, e.g., for $X^d$.
\medskip

The  following assumptions will be in order.
\begin{enumerate} 
\setlength{\itemsep}{1mm}
\item[{\bf A1}] 
\luca{$\alpha \in [0, d-1]$} is a fixed number and
$\rho : [0, +\infty) \to [0, +\infty)$ 
is of class $C^1$ such that 
\[
s \mapsto |\rho'(s)|s^{d-1-\alpha} \in L^1(0, + \infty),
\qquad
\luca{s \mapsto \rho(s)s^{d-2-\alpha} \in L^1(0, + \infty),}
\]
with the renormalization 
\[
\int_{0}^{+\infty} \rho(s)s^{d+1-\alpha}\,ds = c_{d}, 
\qquad\text{where}\qquad
c_{d} := \dfrac{2}
               {\int_{S^{d-1}}|e_{1}\cdot\sigma|^2\,d{\cal H}^{d-1}(\sigma)}.
\]
The family of mollifiers $\{\rho_{\ep}\}_{\ep>0}$ is defined, for every $\ep>0$, as 
\[
\rho_{\ep}(r) := \frac{1}{\ep^d}\rho\left(\frac{r}{\ep}\right) 
\quad \mbox{for}\ r \geq 0.
\]
The family of convolution kernels $\{J_\ep\}_{\ep>0}$
is then defined, for every $\ep>0$, as
\[
J_{\ep} : \mathbb{R}^{d} \to \mathbb{R}, \qquad
J_{\ep}(z) := \dfrac{\rho_{\ep}(|z|)}{\ep^{2-\alpha}|z|^{\alpha}}, 
\quad z \in \mathbb{R}^{d}.
\]
Noting that $\{J_{\ep}\}_{\ep>0}\subset W^{1, 1}(\mathbb{R}^{d})$ 
(see \luca{the proof of} 
\cite[Lemma 3.1]{DST_JDE2021}), it is well defined, for every $\ep>0$,
the nonlocal energy contribution $E_\ep:H\to[0,+\infty)$ as
\begin{equation}
  E_\ep(u):=\frac14
  \int_{\Omega}\int_{\Omega} 
  J_{\ep}(x-y)|u(x) - u(y)|^2\,dxdy, \quad u\in H,
\label{pier1}
\end{equation}
as well as the convolution term $a_\ep:=J_\ep*1$ and
the nonlocal (linear) operator $B_\ep:H\to H$ as
\begin{equation}
  B_\ep u:=a_\ep u - J_\ep*u, \quad u\in H. \label{pier2}
\end{equation}

\item[{\bf A2}] $\beta : \mathbb{R} \to \mathbb{R}$                                
is non-decreasing and locally Lipschitz continuous.
We denote by $\widehat{\beta} : \mathbb{R} \to [0, +\infty)$ 
the unique $C^1$ convex function such that 
$\widehat{\beta}(0) = 0$ and $\widehat\beta'=\beta$: 
as usual,  we identify
$\beta$ with the subdifferential $\partial\widehat\beta$
in the sense of monotone analysis. \luca{We further assume that there are
an exponent $q>1$ and a constant $c_{\beta} > 0$ such that 
\[
  |\beta(r)|^q \leq c_{\beta}\bigl( 1 + \widehat\beta(r) \bigr)
\quad\forall\,r\in\mathbb R,
\]
where also $q\geq 6/5$ if $d=3$.}

\item[{\bf A3}] $\pi : \mathbb{R} \to \mathbb{R}$ 
is Lipschitz continuous. 

\item[{\bf A4}] $f \in L^2(0, T; H)$.
                 
\item[{\bf A5}] The initial data satisfy 
\begin{align*}
&( \theta_0, \varphi_0, v_0) \in V\times V\times H,\\
&\{(\theta_{0,\ep}, \varphi_{0,\ep}, v_{0,\ep})\}_{\ep>0} \subset 
V\times L^\infty(\Omega) \times L^\infty(\Omega).
\end{align*}
Moreover, there exists a constant $c_{1} > 0$ such that
\[
\|\theta_{0,\ep}\|_V^2 + \|\varphi_{0, \ep}\|_H^2 +  E_\ep(\varphi_{0,\ep})
 + \|\widehat{\beta}(\varphi_{0, \ep})\|_{L^1(\Omega)}
 +\|v_{0,\ep}\|_{H}^2\leq c_{1}
 \quad\forall\,\ep>0,
\]
and it holds that
	\begin{align*}
	\lim_{\ep\searrow0}\left[
	\|\theta_{0,\ep}-\theta_0\|_H^2 + 
	\|\varphi_{0,\ep}-\varphi_0\|_H^2
	+\|v_{0,\ep}-v_0\|_{V^*}^2
	\right]=0.
	\end{align*}  
\end{enumerate}

\begin{remark}
  Concerning assumption {\bf A1}, 
  the reader may compare with the analogous requirements in
  \cite{ponce} and \cite{DST_JDE2021}.
  \luca{In particular, it is immediate to check that the 
  generalisation to $\alpha\in[0,d-1]$ follows from the proof
  of \cite[Lemma 3.1]{DST_JDE2021} by using the integrability 
  requirements on $\rho$ in  {\bf A1}.} Next, let  
  us comment on assumption {\bf A2}. 
  First of all, we recall the embeddings
  $V\hookrightarrow L^\infty(\Omega)$ in $d=1$, 
  $V\hookrightarrow L^p(\Omega)$ for all $p\in[1,+\infty)$ in $d=2$, 
  and $V\hookrightarrow L^6(\Omega)$ in $d=3$. 
  Hence, it readily holds that
  \[
  L^q(\Omega)\subset V^*
  \qquad\text{in every space dimension } d=1,2,3,
  \]
  \luca{Secondly, let us briefly discuss 
  the growth condition in {\bf A2}: note that $\beta$ can have any polynomial
  growth at infinity in dimensions $d=1,2$, and any polynomial growth up to 
  $5$-th order in dimension $d=3$.}
  In particular, we note that the classical fourth order double-well potential, 
  leading to $\beta (r) = r^3$ and $\pi (r) = - r$, with $r\in \mathbb{R}$,  
  satisfies {\bf A2}--{\bf A3} in every space dimension.
\end{remark}

\medskip

We define solutions of the problems \ref{Pep} and 
weak solutions of \ref{P} as follows. 
%
%
%
 \begin{df}      \label{sol:nonloc}   
 Let $\ep>0$.
 A {\it solution} of problem \ref{Pep} is a pair $(\theta_{\ep}, \varphi_{\ep})$ with 
    \begin{align*}
    \theta_{\ep} &\in H^1(0, T; H) 
                                   \cap L^{\infty}(0, T; V) \cap L^{2}(0, T; W),  
    \\
    \varphi_{\ep} &\in W^{2, \infty}(0, T; H) 
                                   \cap W^{2, 2}(0, T; L^{\infty}(\Omega)) \ 
                                       {}( \subset W^{1, \infty}(0, T; L^{\infty}(\Omega))\,){},\\
    \beta(\varphi_\ep) &\in L^\infty(0,T; H\cap L^q(\Omega)),
    \end{align*}
 such that
    \begin{alignat*}{2}
        &(\theta_{\ep})_{t} + (\varphi_{\ep})_{t} - \Delta\theta_{\ep} = f 
                           \quad&&\mbox{a.e.\ in}\ \Omega\times(0, T), 
     \\[2mm]
        &(\varphi_{\ep})_{tt} + (\varphi_{\ep})_{t} 
           + B_\ep\varphi_\ep
           + \beta(\varphi_{\ep}) + \pi(\varphi_{\ep}) 
         = \theta_{\ep}
                \quad&&\mbox{a.e.\ in}\ \Omega\times(0, T), 
     \\[2mm]
        &\theta_{\ep}(0) = \theta_{0,\ep},\ 
        \varphi_{\ep}(0) = \varphi_{0, \ep},\ 
        (\varphi_{\ep})_{t}(0) = v_{0,\ep} 
                       \quad&&\mbox{in}\ H. 
     \end{alignat*}
 \end{df}
%
%
%
 \begin{df}         \label{sol:loc}   
 A {\it weak solution} of problem \ref{P} is a pair $(\theta, \varphi)$ with
    \begin{align*}
    \theta &\in H^1(0, T; H) 
                                \cap L^{\infty}(0, T; V) \cap L^{2}(0, T; W), 
    \\ 
    \varphi &\in W^{2, \infty}(0, T; V^{*}) 
                               \cap W^{1, \infty}(0, T; H) 
                                  \cap L^{\infty}(0, T; V),\\
    \beta(\varphi) &\in L^\infty(0,T; L^q(\Omega)),
    \end{align*}
 such that
    \begin{alignat*}{2}
        &\theta_{t} + \varphi_{t} - \Delta\theta = f  
        \quad&&\mbox{a.e.\ in}\ \Omega\times(0, T),  
     \\[2mm]
        &\langle \varphi_{tt}, w \rangle_{V^{*}, V} 
           + (\varphi_{t}, w)_{H} 
           + \int_\Omega \nabla\varphi\cdot \nabla w
           \\[1mm]
          &\hspace{5mm}
           + \langle \beta(\varphi), 
                             w \rangle_{V^*, V} 
           + (\pi(\varphi), w)_{H} 
          = (\theta, w)_{H}
           \quad&&\mbox{for all}\ w \in V, \ \mbox{a.e.\ in}\ (0, T),
     \\[2mm]
        &\theta(0) = \theta_{0},\ 
        \varphi(0) = \varphi_{0},\ 
        \varphi_{t}(0) = v_{0} 
                       \quad&&\mbox{in}\ H \mbox{ or } V^*. 
     \end{alignat*}
 \end{df}

\begin{remark}
  Note that the variational formulation in Definition~\ref{sol:loc}   
  makes sense since in our framework it holds that $L^q(\Omega)\subset V^*$
  for dimensions $1,2,3$.
\end{remark}

Existence and uniqueness of solutions for the nonlocal problem \ref{Pep}
have been obtained in \cite{K7} in a very general  framework.
In particular, we  have the following.
\begin{thm}[\cite{K7}]\label{maintheorem1}
Assume {\rm {\bf A1}-{\bf A5}}. 
Then, 
for all $\ep \in (0, 1)$ 
there exists a unique solution 
$(\theta_{\ep}, \varphi_{\ep})$ of {\rm \ref{Pep}}. 
\end{thm}

The main result of the paper is the convergence of the solutions 
of the nonlocal problem \ref{Pep} to those of the local one \ref{P}.
As a byproduct, this shows also existence of solutions to the 
local problem \ref{P}.
\begin{thm}\label{maintheorem2}
Assume {\rm {\bf A1}--{\bf A5}}
and let $\{(\theta_\ep, \varphi_\ep)\}_{0<\ep<1}$ be the family
of solutions  of the nonlocal problems \{{\rm \ref{Pep}}\}$_{0<\ep<1}$.
Then, there exists a weak solution $(\theta,\varphi)$ of the local 
problem {\rm\ref{P}} and a sequence $\{\ep_j\}_{j\in\mathbb N}$
such that, as $j\to\infty$, $\ep_j\searrow0$ and
\begin{align*}
\theta_{\ep_j} \to \theta 
\quad&\mbox{weakly in } H^1(0, T; H) 
\cap L^2(0, T; W),
\\
\theta_{\ep_j} \to \theta 
\quad&\mbox{weakly* in } L^{\infty}(0, T; V)  \mbox{ and strongly in } C^0 ([0, T]; H),
\\
\varphi_{\ep_j} \to \varphi 
\quad&\mbox{strongly in } C^0([0, T]; H), 
\\
(\varphi_{\ep_j})_{t} \to \varphi_{t} 
\quad&\mbox{weakly* in } L^\infty(0, T; H)  \mbox{ and strongly in } C^0([0, T]; V^*),
\\
(\varphi_{\ep_j})_{tt} \to \varphi_{tt} 
\quad&\mbox{weakly* in } L^{\infty}(0, T; V^*), 
\\
B_{\ep_j}\varphi_{\ep_j} \to B\varphi 
\quad&\mbox{weakly* in } L^{\infty}(0, T; V^*), 
\\
\beta(\varphi_{\ep_j}) \to \beta (\varphi) 
\quad&\mbox{weakly* in } L^{\infty}(0, T; L^q(\Omega)), 
\end{align*}
where the operator $B:V\to V^*$ is specified by 
\begin{equation}
\langle Bu_1, u_2 \rangle_{V^{*}, V} 
= \int_{\Omega} \nabla u_1(x)\cdot\nabla u_2(x)\,dx 
\quad \mbox{for } u_1,u_2 \in V.
\label{pier4}
\end{equation}
\end{thm}

\section{Preliminaries}
\label{Sec:prelim}
We recall here some properties of the family of convolution kernels 
and the main results on the behaviour of the nonlocal terms as $\ep\searrow0$:
therein let us refer to~\cite{DST_JDE2021} for the proofs.
These results will be used later for the proof of Theorem~\ref{maintheorem2}.  

We  define the local energy  contribution $E:H\to[0,+\infty]$ as
\begin{equation}
E(\varphi) := 
\begin{cases}
\frac{1}{2}\int_{\Omega} |\nabla\varphi|^2  & \varphi \in V, 
\\ 
+ \infty  & \varphi \in H \setminus V. 
\end{cases}
\label{pier3}
\end{equation}
Let us point out that the local operator $B:V\to V^*$ introduced in \eqref{pier4}
is the subdifferential of $E_{|_V}$,  i.e., $B$ provides a 
weak formulation of the negative Laplacian with Neumann homogeneous boundary conditions.

The lemma below provides a characterization of the Fr\'echet derivative of the functional~$E_\ep $ defined by \eqref{pier1} as the operator $B_\ep$ in \eqref{pier2}.

\begin{lem}[{\cite[Lemma 3.2]{DST_JDE2021}}]\label{L1}
For all $\ep > 0$, the nonlocal energy 
$E_{\ep} : H \to [0, +\infty)$ is of class $C^1$ 
and its Fr\'echet derivative  $DE_{\ep}$ coincides with the operator $B_\ep : H \to H$.
In particular, it holds, for all $u_1,u_2 \in H$, that
\begin{align*}
(DE_{\ep}(u_1), u_2)_{H} 
&= (B_{\ep}u_1, u_2)_{H}  = 
\int_{\Omega}
       (a_{\ep}(x)u_1(x) - (J_{\ep}\ast u_1)(x))u_2(x)\,dx 
\\
&= \frac{1}{2}
      \int_{\Omega}\int_{\Omega}
        J_{\ep}(x-y)(u_1(x)-u_1(y))(u_2(x)-u_2(y))\,dxdy.
\end{align*}
\end{lem}

Some convergence properties of $E_\ep $ and $B_\ep$ as $\ep \searrow 0$ are stated in the 
following two results. We recall that the limit functional $E$ and operator $B$ are defined in \eqref{pier3} and \eqref{pier4}, respectively.

\begin{lem}[{\cite[Lemma 3.3]{DST_JDE2021}}]\label{L2}
For all $v_1,v_2 \in V$ 
it holds that 
\begin{align*}
&\lim_{\ep \searrow 0} E_{\ep}(v_1) = E(v_1), 
\\
&\lim_{\ep \searrow 0} 
  (B_\ep v_1,v_2)_H 
  = \langle Bv_1, v_2 \rangle_{V^{*}, V} .
\end{align*}
Moreover, for all sequence $\{u_\ep\}_{\ep>0} \subset H$ 
and all $u \in H$, 
it holds that 
\begin{align*}
&\sup_{\ep > 0} E_{\ep}(u_{\ep}) < +\infty  
\quad \Rightarrow \quad 
\{u_{\ep}\}_{\ep>0}\ \mbox{is relatively compact in}\ H,
\\[1mm]
&u_{\ep} \to u \quad \mbox{in}\ H \ 
\mbox{as}\ \ep \searrow 0
\quad \Rightarrow \quad 
E(u) \leq \liminf_{\ep \searrow 0} E_{\ep}(u_{\ep}). 
\end{align*}
\end{lem}

\begin{lem}[{\cite[Proposition 3.1]{DST_JDE2021}}]\label{L3}
Let $\{u_{\ep}\}_{\ep>0} \subset H$ 
and $u \in H$ be such that 
$u_{\ep} \to u$ in $H$ 
and $\sup_{\ep>0}E_{\ep}(u_{\ep})<+\infty$.
Then, 
$u \in V$ and 
$B_{\ep}(u_{\ep}) \to B(u)$ weakly in $V^*$. 
\end{lem}

\begin{remark}\label{remaboutBepEep}
We note that there exists a constant $C>0$ such that,
for every $u\in H$, $v\in V$, and $\ep>0$, it holds that 
\begin{align*}
\langle B_{\ep}(u), v \rangle_{V^{*}, V} 
&= (DE_{\ep}(u), v)_{H} 
\\
&= \frac{1}{2}
      \int_{\Omega}\int_{\Omega}
        J_{\ep}(x-y)(u(x)-u(y))(v(x)-v(y))\,dxdy 
\\
&\leq 2\sqrt{E_{\ep}(u)}\sqrt{E_{\ep}(v)}
\\
&\leq C\sqrt{E_{\ep}(u)}\, \|v\|_{V},
\end{align*}
where the last step is a direct consequence 
of \cite[Theorem 1]{BBM}. This readily implies  that 
\[
  \|B_\ep(u)\|_{V^*}\leq C\sqrt{E_\ep(u)} \quad\forall\,u\in H, \quad\forall\,\ep>0.
\]
\end{remark}

The following statement expresses a technical property that will be very useful for our limit procedure.

\begin{lem}[{\cite[Lemma 3.4]{DST_JDE2021}}]\label{L4}
For all $\delta > 0$ there exist constants $C_{\delta} > 0$ 
and $\ep_{\delta} > 0$ such that,
for all sequence 
$\{u_{\ep}\}_{\ep \in (0, \ep_{\delta})} \subset H$,
it holds for all $\ep_{1}, \ep_{2} \in (0, \ep_{\delta})$ that
\[
\|u_{\ep_{1}} - u_{\ep_{2}}\|_{H}^2 
\leq \delta\bigl(E_{\ep_{1}}(u_{\ep_{1}}) + E_{\ep_{2}}(u_{\ep_{2}})\bigr) 
      + C_{\delta}\|u_{\ep_{1}} - u_{\ep_{2}}\|_{V^{*}}^2.
\]
\end{lem}


\section{Proof of Theorem \ref{maintheorem2}}\label{sec:proof}
In this section we prove Theorem~\ref{maintheorem2}:
the proof is split into two parts. Firstly, we derive estimates on the 
solutions of the nonlocal problem, independently of $\ep$. Secondly, 
we  pass to the limit as $\ep\searrow0$ and recover a weak 
solution of the local problem.
Since we are interested in the behaviour as $\ep\searrow0$, as pointed out 
in the statement of Theorem~\ref{maintheorem2}
we restrict here to $\ep\in(0,1)$ with no loss of generality.

\subsection{Uniform estimates for \ref{Pep}}\label{ssec:est}
We collect the estimates for the nonlocal problem \ref{Pep}
in the following lemmata.
\begin{lem}\label{esti1}
There exists a constant $C>0$ such that, for all $\ep\in(0,1)$, it holds
\begin{align*}
\|\theta_{\ep}\|_{L^{\infty}(0, T; H)}^2 +
\|\nabla\theta_{\ep}\|_{L^{2}(0, T; H)}^2 &\leq C,\\
\|\varphi_{\ep}\|_{L^\infty(0, T; H)}^2 + \|(\varphi_{\ep})_t\|_{L^\infty(0, T; H)}^2 
+ \|E_{\ep}(\varphi_{\ep})\|_{L^{\infty}(0, T)} &\leq C,\\
\|\widehat{\beta}(\varphi_{\ep})\|_{L^{\infty}(0, T; L^1(\Omega))} &\leq C .
\end{align*}
\end{lem}
\begin{proof}
Testing the first equation in \ref{Pep} by $\theta_{\ep}$, 
after integration over $(0, t)$, where $t \in [0, T]$, we obtain
\begin{align*}
&\frac{1}{2}\|\theta_{\ep}(t)\|_{H}^2 
+ \int_{0}^{t} \|\nabla\theta_{\ep}(s)\|_{H}^2\,ds  \\  
&= \frac{1}{2}\|\theta_{0,\ep}\|_{H}^2 
   + \int_{0}^{t} (f(s), \theta_{\ep}(s))_{H}\,ds 
   - \int_{0}^{t} ((\varphi_{\ep})_{t}(s), \theta_{\ep}(s))_{H}\,ds. 
\end{align*}
Moreover, we test the second equation in \ref{Pep} by $(\varphi_{\ep})_{t}$
and integrate over $(0, t)$
to infer that 
\begin{align*}
&\frac{1}{2}\|(\varphi_{\ep})_{t}(t)\|_{H}^2 
+ \int_{0}^{t} \|(\varphi_{\ep})_{t}(s)\|_{H}^2\,ds 
+ E_{\ep}(\varphi_{\ep}(t)) 
+ \int_{\Omega} \widehat{\beta}(\varphi_{\ep}(t)) 
 \\
&= \frac{1}{2}\|v_{0,\ep}\|_{H}^2 
     + E_{\ep}(\varphi_{0, \ep}) 
     + \int_{\Omega} \widehat{\beta}(\varphi_{0, \ep}) 
     - \int_{0}^{t} (\pi(\varphi_{\ep}(s)), (\varphi_{\ep})_{t}(s))_{H}\,ds 
 \\ 
&\,\quad + \int_{0}^{t} ((\varphi_{\ep})_{t}(s), \theta_{\ep}(s))_{H}\,ds.
\end{align*}
Summing up the two equalities together with  
\[
  \frac{1}{2}\|\varphi_{\ep}(t)\|_{H}^2 
= \frac{1}{2}\|\varphi_{0, \ep}\|_{H}^2 
   + \int_{0}^{t} ((\varphi_{\ep})_{t}(s), \varphi_{\ep}(s))_{H}\,ds
\]
yields then 
\begin{align}
\notag
&\frac{1}{2}\|\theta_{\ep}(t)\|_{H}^2 + \frac{1}{2}\|\varphi_{\ep}(t)\|_{H}^2 +
\frac{1}{2}\|(\varphi_{\ep})_{t}(t)\|_{H}^2 +
E_{\ep}(\varphi_{\ep}(t)) 
+ \int_{\Omega} \widehat{\beta}(\varphi_{\ep}(t))\\
\notag
&\qquad+\int_{0}^{t} \|\nabla\theta_{\ep}(s)\|_{H}^2\,ds 
+\int_{0}^{t} \|(\varphi_{\ep})_{t}(s)\|_{H}^2\,ds 
 \\  
\notag
&= \frac{1}{2}\|\theta_{0,\ep}\|_{H}^2 
+ \frac{1}{2}\|\varphi_{0, \ep}\|_{H}^2 +\frac{1}{2}\|v_{0,\ep}\|_{H}^2 
     + E_{\ep}(\varphi_{0, \ep}) 
     + \int_{\Omega} \widehat{\beta}(\varphi_{0, \ep}) \\
&\qquad+ \int_{0}^{t} (f(s), \theta_{\ep}(s))_{H}\,ds 
+ \int_{0}^{t} (\varphi_{\ep}(s)-\pi(\varphi_{\ep}(s)), (\varphi_{\ep})_{t}(s))_{H}\,ds.
\label{a1}
\end{align}
Now, the first five terms on the right-hand side are uniformly bounded in $\ep$
thanks  to assumption {\bf A5} on the initial data.
Moreover, by virtue of assumption {\bf A4} on $f$ 
one has, thanks to the Young inequality, that
\[
   \int_{0}^{t} (f(s), \theta_{\ep}(s))_{H}\,ds \leq
 \frac12\|f\|^2_{L^2(0,T; H)} + \frac12\int_0^t\|\theta_\ep(s)\|_H^2\,ds.
\]
Furthermore, the
assumption {\bf A3} on the Lipschitz continuity of $\pi$
and the Young inequality yield, for some constant $C$ independent of $\ep$,
\begin{align*}
& \int_{0}^{t} (\varphi_{\ep}(s)-\pi(\varphi_{\ep}(s)), (\varphi_{\ep})_{t}(s))_{H}\,ds \leq
 C\int_0^t(1+\|\varphi_\ep(s)\|_H)\|(\varphi_\ep)_t(s)\|_H\,ds\\
 &\qquad{}\leq C\int_0^t\|(\varphi_\ep)_t(s)\|_H^2\,ds
 +\frac{C}2\int_0^t(1+\|\varphi_\ep(s)\|_H^2)\,ds.
\end{align*}
Taking this information into account in the inequality \eqref{a1},
by possibly updating the value of the constant $C$, independently of $\ep$,
we obtain that 
\begin{align*}
&\frac{1}{2}\|\theta_{\ep}(t)\|_{H}^2 + \frac{1}{2}\|\varphi_{\ep}(t)\|_{H}^2+
\frac{1}{2}\|(\varphi_{\ep})_{t}(t)\|_{H}^2 +
E_{\ep}(\varphi_{\ep}(t)) 
+ \int_{\Omega} \widehat{\beta}(\varphi_{\ep}(t))\\
&\qquad+\int_{0}^{t} \|\nabla\theta_{\ep}(s)\|_{H}^2\,ds 
+\int_{0}^{t} \|(\varphi_{\ep})_{t}(s)\|_{H}^2\,ds 
 \\  
&\leq C\left(1+\int_0^t\|\varphi_\ep(s)\|_H^2\,ds +
\int_{0}^{t} \|(\varphi_{\ep})_{t}(s)\|_{H}^2\,ds \right),
\end{align*}
and thanks to the Gronwall lemma we conclude.
\end{proof}

\begin{lem}\label{esti2}
There exists a constant $C>0$ such that, for all $\ep\in(0,1)$, it holds
\begin{equation}
\|(\theta_{\ep})_t\|_{L^2(0, T; H)}^2 
+ \|\theta_{\ep}\|_{L^{\infty}(0, T; V)}^2
+\|\Delta \theta_{\ep}\|_{L^2(0, T; H)}
\leq C . 
\label{pier5}
\end{equation}
\end{lem}
\begin{proof}
We test the first equation in \ref{Pep} by $(\theta_{\ep})_t$ and
integrate over $(0, t)$. With the help of the Young inequality we deduce that
\begin{align*}
  &\int_0^t\|(\theta_\ep)_t(s)\|_H^2\,ds
  +\frac12\|\nabla\theta_\ep(t)\|_H^2\\
  &\quad=\frac12\|\nabla\theta_{0,\ep}\|_H^2
  +\int_0^t(f(s)-(\varphi_\ep)_t(s), (\theta_\ep)_t(s))_H\,ds\\
  &\quad\leq \frac12\|\nabla\theta_{0,\ep}\|_H^2
  +\frac12\int_0^t\|(\theta_\ep)_t(s)\|_H^2\,ds
  +\|f\|^2_{L^2(0,T; H)} + \|(\varphi_\ep)_t\|^2_{L^2(0,T; H)}.
\end{align*}
Rearranging the terms, the estimate \eqref{pier5} for the first two terms follows 
by virtue of assumptions~{\bf A4--A5} and Lemma~\ref{esti1}. Eventually, 
the estimate for the third term in \eqref{pier5} is the consequence of a comparison 
in the first equation of \ref{Pep}.
\end{proof}

\begin{lem}\label{esti4}
Let $q$ be as in assumption~{\bf A2}. Then
there exists a constant $C>0$ such that, for all $\ep\in(0,1)$, it holds
\begin{equation}
\|\beta(\varphi_{\ep})\|_{L^{\infty}(0, T; L^q(\Omega))} \leq C . \label{pier6}
\end{equation}
\end{lem}
\begin{proof}
\luca{By assumption {\bf A2} one has that 
\begin{equation}
  |\beta(\varphi_\ep)|^q \leq c_\beta\left(1+\widehat\beta(\varphi_\ep)\right),
\label{pier7}
\end{equation}
where the right-hand side is bounded in $L^\infty(0,T; L^1(\Omega))$
as a consequence of Lemma~\ref{esti1}, hence \eqref{pier6} follows.}
\end{proof}

\begin{lem}\label{esti5}
There exists a constant $C>0$ such that, for all $\ep\in(0,1)$, it holds
\begin{equation}
\|(\varphi_{\ep})_{tt}\|_{L^{\infty}(0, T; V^*)}
+\|B_\ep\varphi_\ep\|_{L^\infty(0,T; V^*)} \leq C.
\label{pier8}
\end{equation}
\end{lem}
\begin{proof}
Recalling Remark~\ref{remaboutBepEep}, from Lemma~\ref{esti1} one has 
straightaway that 
\[
  \|B_\ep\varphi_\ep\|_{L^\infty(0,T; V^*)} \leq C.
\]
Moreover, recalling that $L^q(\Omega)\hookrightarrow V^*$ (cf.~ assumption {\bf A2}),
by comparison in the second equation of~\ref{Pep} one infers that
\[
  \|(\varphi_{\ep})_{tt}\|_{L^{\infty}(0, T; V^*)}\leq C
\]
owing to the estimates in Lemmas~\ref{esti1}--\ref{esti4}, and this yields \eqref{pier8}.
\end{proof}

\subsection{Passage to the limit}
We are now ready to pass to the limit as $\ep\searrow0$ in the problem~\ref{Pep} and complete the proof of 
Theorem~\ref{maintheorem2}.

\begin{proof}[Proof of Theorem~\ref{maintheorem2}]
By recalling Lemmas \ref{esti1}-\ref{esti5}, the compact embeddings 
$W \hookrightarrow V \hookrightarrow H \hookrightarrow V^*$, and
the Aubin--Lions compactness results (see, e.g., \cite[Sect.~8, Cor.~4]{Sim87}), 
it is straightforward to infer that there exist some functions 
$\theta$, $\varphi$, $\eta$, $\xi$ 
with 
\begin{align*}
&\theta \in H^1(0, T; H) 
                                \cap L^{\infty}(0, T; V) \cap L^{2}(0, T; W), 
\\ 
&\varphi \in W^{2, \infty}(0, T; V^{*}) 
                               \cap W^{1, \infty}(0, T; H), 
\\
&\eta \in L^{\infty}(0, T; V^{*}), 
\\
&\xi \in L^{\infty}(0, T; L^{q}(\Omega)),    
\end{align*}
and a sequence $\{\ep_j\}_{j\in\mathbb N}$ such that, 
as $j\to\infty$, $\ep_j\searrow0$ and
\begin{align}
\theta_{\ep_j} \to \theta 
\quad&\mbox{weakly in } H^1(0, T; H)  \cap L^2(0, T; W), 
\label{weak1}
\\
\theta_{\ep_j} \to \theta 
\quad&\mbox{weakly* in } L^\infty(0,T; V), 
\label{weak1'}
\\
\theta_{\ep_j} \to \theta
\quad&\mbox{strongly in } C^0([0, T]; H)\cap L^2(0,T;V),
\label{strong1}
\\
\varphi_{\ep_j} \to \varphi
\quad&\mbox{weakly* in }  W^{2, \infty}(0, T; V^{*}) \cap W^{1, \infty}(0, T; H),
\label{weak2}
\\
\varphi_{\ep_j} \to \varphi 
\quad&\mbox{strongly in } C^1([0, T]; V^*), 
\label{strong2}
\\
B_{\ep_j}\varphi_{\ep_j} \to \eta 
\quad&\mbox{weakly* in } L^{\infty}(0, T; V^*), 
\label{weak4}
\\
\beta(\varphi_{\ep_j}) \to \xi 
\quad&\mbox{weakly* in } L^{\infty}(0, T; L^q(\Omega)).
\label{weak5}
\end{align}
Now, from Lemmas \ref{L4} and \ref{esti1}, it follows that 
for every $\delta>0$ there exist $c_\delta>0$ and $\ep_\delta>0$
such that, for all $i,j\in\mathbb N$ with $\ep_{i}, \ep_j<\ep_\delta$, one has
\[
  \|\varphi_{\ep_i}-\varphi_{\ep_j}\|_{C^0([0,T]; H)}^2
  \leq C\delta + c_\delta\|\varphi_{\ep_i}-\varphi_{\ep_j}\|_{C^0([0,T]; V^*)}^2.
\]
Hence, by  arbitrariness of $\delta$ and the  strong convergence 
\eqref{strong2} we also deduce that, as $j\to\infty$,
\begin{equation}\label{strong4}
\varphi_{\ep_j} \to \varphi 
\quad \mbox{strongly in } C^0([0, T]; H).
\end{equation}
This implies, by the Lipschitz continuity of $\pi$ assumed in {\bf A3}, that
\begin{equation}\label{strong5}
\pi(\varphi_{\ep_j}) \to \pi(\varphi) 
\quad \mbox{strongly in } C^0 ([0, T]; H).
\end{equation}

Now, from Lemma~\ref{L2}, the strong convergence \eqref{strong4},
and Lemma~\ref{esti1} we have
\[
  E(\varphi(t))\leq\liminf_{j\to\infty}E_{\ep_j}(\varphi_{\ep_j}(t)) \leq C \quad\forall\,t\in[0,T],
\]
which implies in turn, by definition of $E$, that 
\begin{equation}\label{phiV}
\varphi\in L^\infty(0,T; V) \qquad\text{and}\qquad
\varphi(t)\in V\quad\forall\,t\in[0,T].
\end{equation}
Furthermore, for every $t\in[0,T]$, one has that 
$\varphi_{\ep_j}(t)\to\varphi(t)$ in $H$ by \eqref{strong4} and 
$\sup_{j\in\mathbb N} \|E_{\ep_j}(\varphi_{\ep_j})\|_{L^\infty (0,T)} \leq C<+\infty$ by Lemma~\ref{esti1}:
hence, Lemma~\ref{L3} ensures that, as $j\to\infty$,
\begin{equation}\label{b1}
B_{\ep_j}\varphi_{\ep_j}(t) \to B\varphi(t)
\quad \mbox{weakly in } V^*, \quad\mbox{for a.a.}\ t \in (0, T).
\end{equation}
We claim that \eqref{b1} implies, thanks to \eqref{weak4}, that
\begin{equation}\label{etaBvarphi}
\eta(t) = B\varphi(t) \quad\mbox{in } V^{*}, \quad
\mbox{for a.a.}\ t \in (0, T).
\end{equation}
Indeed, let $\psi \in C_{c}^{\infty}(\Omega\times(0, T))$. 
Then, by the convergence \eqref{b1} one has that
\[
  \langle B_{\ep_j}\varphi_{\ep_j}(t), F^{-1}\psi(t) \rangle_{V^{*}, V} 
\to \langle B\varphi(t), F^{-1}\psi(t) \rangle_{V^{*}, V} 
\quad \mbox{for a.a.}\ t \in (0, T),
\]
while by the Remark~\ref{remaboutBepEep} and
Lemma~\ref{esti1} it holds that
\[
|\langle B_{\ep_j}\varphi_{\ep_j}(t), F^{-1}\psi(t) \rangle_{V^{*}, V}| 
 \leq C\|\psi\|_{L^\infty(0,T; V^*)} \quad\mbox{for a.a.}\ t \in (0, T).
\]
Hence, it follows from 
the Lebesgue dominated convergence theorem 
that 
\begin{equation}\label{b4}
\lim_{j\to\infty}
  \int_{0}^{T}
    \left|\langle B_{\ep_j}\varphi_{\ep_j}(t), F^{-1}\psi(t) \rangle_{V^{*}, V} 
           - \langle B\varphi(t), F^{-1}\psi(t) \rangle_{V^{*}, V}\right|\,dt 
= 0.
\end{equation}
Combining \eqref{weak4} and \eqref{b4} implies then that, as $j\to\infty$,
\begin{align*}
&\left|
  \int_{\Omega\times(0, T)} 
               (F^{-1}(B\varphi - \eta))(x, t)\psi(x, t)\,dxdt  
  \right|
\\
&= \left|
     \int_{0}^{T} (\psi(t), F^{-1}(B\varphi(t) - \eta(t)))_{H}\,dt 
    \right| 
= \left|
      \int_{0}^{T} 
         \langle \psi(t), F^{-1}(B\varphi(t) - \eta(t)) \rangle_{V^{*}, V}\,dt 
   \right| 
\\
&= \left|
      \int_{0}^{T} 
        \langle B\varphi(t) - \eta(t), F^{-1}\psi(t) \rangle_{V^{*}, V}\,dt 
         \right| 
\\
&\leq \left|
      \int_{0}^{T} 
        \langle B\varphi(t) - B_{\ep_j}\varphi_{\ep_j}(t), 
                                                  F^{-1}\psi(t) \rangle_{V^{*}, V}\,dt 
         \right| 
\\ 
      &\,\quad + \left|
      \int_{0}^{T} 
        \langle B_{\ep_j}\varphi_{\ep_j}(t) - \eta(t), 
                                            F^{-1}\psi(t) \rangle_{V^{*}, V}\,dt 
         \right| 
\to 0 .
\end{align*}
Hence we can conclude that 
\begin{equation*}
\int_{\Omega\times(0, T)} (F^{-1}(B\varphi - \eta))(x, t)\psi(x, t)\,dxdt 
= 0 \quad \forall\,\psi \in C_{c}^{\infty}(\Omega\times(0, T)),
\end{equation*}
whence, by the arbitrariness of $\psi$, we deduce that
\begin{equation*}
F^{-1}(B\varphi - \eta) = 0 
\quad \mbox{a.e.\ in}\ \Omega\times(0, T).
\end{equation*}
By definition of $F$, this leads to \eqref{etaBvarphi}, as desired.

Eventually, we are left with verifying that 
\begin{equation}\label{xibeta}
\xi = \beta(\varphi) \quad \mbox{a.e.\ in}\ \Omega\times(0, T).
\end{equation}
To this end, note that since $\beta$ is single-valued and continuous,
from the strong convergence \eqref{strong4} it holds that, at least for a subsequence,
\[
  \beta(\varphi_{\ep_j})\to \beta(\varphi) \quad \mbox{a.e.\ in}\ \Omega\times(0, T).
\]
By the Severini--Egorov theorem and the boundedness property in \eqref{pier6}
we infer~that (the reader may see \cite[Lemme~1.3, p.12]{Lio69} for some detail)
\[
  \beta(\varphi_{\ep_j})\to \beta(\varphi) \quad \text{weakly in } L^q(\Omega\times(0,T))
\]
and 
\[
  \beta(\varphi_{\ep_j})\to \beta(\varphi) \quad \text{strongly in } L^p(\Omega\times(0,T))
  \quad\forall\,p\in[1,q).
\]

These conditions, together with \eqref{weak5}, readily ensure the identification \eqref{xibeta}.

In order to conclude, we only need to let $\ep \searrow0$
in the nonlocal system \ref{Pep}: actually, we do it when $\ep$ takes the values $\ep_j$ 
of the subsequence. To this end, note that the convergences 
proved above allow to pass to the limit directly 
in the first equation and in the equations for the initial data.
As for the second equation in \ref{Pep},
let $\psi \in C_{c}^{\infty}(\Omega\times(0, T))$ be arbitrary. 
Then, testing the second equation of \ref{Pep} by $F^{-1}\psi$ and integrating on $[0,T]$,
we have 
\begin{align*}
0 &= \int_{0}^{T} \langle (\varphi_{\ep_j})_{tt}(t), 
                                   F^{-1}\psi(t) \rangle_{V^{*}, V}\,dt 
     + \int_{0}^{T} ((\varphi_{\ep_j})_{t}(t), F^{-1}\psi(t))_{H}\,dt 
\notag \\
 &\,\quad + \int_{0}^{T} \langle B_{\ep_j}\varphi_{\ep_j}(t), 
                                       F^{-1}\psi(t) \rangle_{V^{*}, V}\,dt 
      + \int_{\Omega\times(0,T)} \beta(\varphi_{\ep_j}(x,t))
                   F^{-1}\psi(x,t)\,dxdt
\notag \\ 
   &\,\quad + \int_{0}^{t} (\pi(\varphi_{\ep_j}(t)), F^{-1}\psi(t))_{H}\,dt 
                - \int_{0}^{t} (\theta_{\ep_j}(t), F^{-1}\psi(t))_{H}\,dt.
\end{align*}
Hence, by letting $j\to\infty$,
we deduce from \eqref{weak1}--\eqref{strong5}, 
\eqref{etaBvarphi}, and \eqref{xibeta} that 
\begin{equation}\label{b9} 
\int_{\Omega\times(0, T)} 
 (F^{-1}(\varphi_{tt} + \varphi_{t} + B\varphi
       + \beta(\varphi) + \pi(\varphi) - \theta))(x, t)\psi(x, t)\,dxdt
= 0
\end{equation}
for all $\psi \in C_{c}^{\infty}(\Omega\times(0, T))$,
which readily implies the variational formulation of problem~\ref{P}.
Therefore, $(\theta,\varphi)$ is a weak solution of~\ref{P}
in the sense of Definition~\ref{sol:loc}, and this concludes the proof.
\end{proof}

\smallskip
\section*{Acknowledgments}
The research of PC has been performed within the framework of the MIUR-PRIN Grant
2020F3NCPX ``Mathematics for industry 4.0 (Math4I4)''. The research of SK is supported 
by JSPS KAKENHI Grant Number JP23K12990. The research of 
LS is part of the activities of 
``Dipartimento di Eccellenza 2023-2027'' of Politecnico di Milano.
Moreover, PC and LS are members of of the GNAMPA (Gruppo Nazionale per l’Analisi Matematica, la Probabilit\`a e le loro Applicazioni) of INdAM (Istituto Nazionale di Alta Matematica).


%
%
%


\begin{thebibliography}{99} 

\bibitem{ab-bos-grass-NLCH}
H.~Abels, S.~Bosia, and M.~Grasselli,
\newblock {C}ahn-{H}illiard equation with nonlocal singular free energies.
\newblock {\em Ann. Mat. Pura Appl. (4)}, 194(4):1071--1106, 2015.

\bibitem{AH}
H.~Abels and C.~Hurm,
\newblock Strong nonlocal-to-local convergence of the {C}ahn-{H}illiard
  equation and its operator.
  Preprint arXiv:2307.02264  [math.AP], 2023, pp.~1--25.

\bibitem{AT}
H.~Abels and Y.~Terasawa,
\newblock Convergence of a nonlocal to a local diffuse interface model for
  two-phase flow with unmatched densities.
\newblock {\em Discrete Contin. Dyn. Syst. Ser. S}, 15(8):1871--1881, 2022.

\bibitem{BBM}
J. Bourgain, H. Brezis, and P. Mironescu, 
{\it Another look at Sobolev spaces}, 
in: Optimal Control and Partial Differential Equations, IOS, Amsterdam, 2001, 
pp.~439--455.

\bibitem{ca-el-sk-DEG}
J.~A. Carrillo, C.~Elbar, and J.~Skrzeczkowski,
\newblock Degenerate {C}ahn-{H}illiard systems: From nonlocal to local.
Preprint arXiv:2303.11929 [math.AP], 2023, pp.~1--30.

\bibitem{cgi}
P.  Colli,  M.  Grasselli and  A.  Ito,
\newblock On  a  parabolic-hyperbolic  Penrose-Fife  phase-field  system.
\newblock {\em Electron. J. Differential Equations} 2002, No. 100, 30 pp. 
(Erratum:  {\em Electron. J. Differential Equations} 2002, No. 100, 32 pp.).

%
\bibitem{DRST}
E.~Davoli, H.~Ranetbauer, L.~Scarpa, and L.~Trussardi,
\newblock Degenerate nonlocal {C}ahn-{H}illiard equations: well-posedness,
  regularity and local asymptotics.
\newblock {\em Ann. Inst. H. Poincar\'{e} C Anal. Non Lin\'{e}aire},
  37(3):627--651, 2020.
  
\bibitem{DRST2023}
E.~Davoli, E.~Rocca, L.~Scarpa, and L.~Trussardi,
\newblock Local asymptotics and optimal control 
for a viscous Cahn-Hilliard-Reaction-Diffusion model for tumor growth.
Preprint arXiv:2311.10457 [math.AP], 2023, pp.~1--30.

\bibitem{DST2021}
E.~Davoli, L.~Scarpa, and L.~Trussardi,
\newblock Nonlocal-to-local convergence of Cahn-Hilliard equations: 
Neumann boundary conditions and viscosity terms.
\newblock{\em Arch.~Ration.~Mech.~Anal.}, 239(1):117--149, 2021. 
%
\bibitem{DST_JDE2021}
E.~Davoli, L.~Scarpa, and L.~Trussardi,
\newblock Local asymptotics for nonlocal convective Cahn-Hilliard equations 
with $W^{1, 1}$ kernel and singular potential. 
\newblock{\em J. Differential Equations}, 289:35--58, 2021.

\bibitem{el-sk-DEG}
C.~Elbar and J.~Skrzeczkowski,
\newblock Degenerate {C}ahn-{H}illiard equation: from nonlocal to local.
\newblock {\em J.~Differential Equations}, 364:576--611, 2023.

\bibitem{frig-gal-gras-DEG}
S.~Frigeri, C.~G. Gal, and M.~Grasselli,
\newblock Regularity results for the nonlocal {C}ahn-{H}illiard equation with
  singular potential and degenerate mobility.
\newblock {\em J.~Differential Equations}, 287:295--328, 2021.

\bibitem{gal-gior-grass-NLCH}
C.~G. Gal, A.~Giorgini, and M.~Grasselli,
\newblock The nonlocal {C}ahn-{H}illiard equation with singular potential:
  well-posedness, regularity and strict separation property.
\newblock {\em J. Differential Equations}, 263(9):5253--5297, 2017.

\bibitem{gal-gior-gras-SEP}
C.~G. Gal, A.~Giorgini, and M.~Grasselli,
\newblock The separation property for 2D {C}ahn-{H}illiard equations: local,
  nonlocal and fractional energy cases.
\newblock {\em Discrete Contin. Dyn. Syst.}, 43(6):2270--2304,
  2023.

\bibitem{gal-grass-NLCH}
C.~G. Gal and M.~Grasselli,
\newblock Longtime behavior of nonlocal {C}ahn-{H}illiard equations.
\newblock {\em Discrete Contin. Dyn. Syst.}, 34(1):145--179, 2014.

\bibitem{GMPZ}
M.~Grasselli, A.~Miranville, V.~Pata, and S.~Zelik, 
\newblock Well-posedness and long time behavior of a parabolic-hyperbolic phase-field system with singular potentials.
\newblock {\em Math. Nachr.}  280(13-14):1475--1509, 2007.

\bibitem{jiang}
J.~Jiang, 
\newblock Convergence to equilibrium for a fully hyperbolic phase-field model with Cattaneo heat flux law.
\newblock {\em Math. Methods Appl. Sci.}  32(9):1156--1182, 2009.

\bibitem{kur1}
S. Kurima, 
\newblock Time discretization of an initial value problem for a simultaneous abstract evolution equation applying to parabolic-hyperbolic phase-field systems.
\newblock
{\em ESAIM Math. Model. Numer. Anal.}  54(3):977--1002, 2020.
		
\bibitem{K7}
S. Kurima,
\newblock Time discretization of a nonlocal phase-field system 
with inertial term.
\newblock {\em Matematiche (Catania)}, 77(1):47--66, 2022.   

\bibitem{kur2}
S. Kurima, 
\newblock Existence for a singular nonlocal phase field system with inertial term. 
\newblock
{\em Acta Appl. Math.} 178, Paper No. 10, 20 pp., 2022.

\luca{\bibitem{kur3}
S. Kurima, 
\newblock Nonlocal to local convergence of singular phase field systems of conserved type.
\newblock
{\em Adv. Math. Sci. Appl.},  31(2): -481--500, 2022. }%
	
\bibitem{Lio69} 
	J.\ L.\ {L}ions, 
	\newblock {\it Quelques m\'ethodes de r\'esolution des probl\`emes aux limites non lin\'eaires}, 
	\newblock Dunod Gauthier-Villas, Paris, 1968.

\bibitem{MRT18}
S.~Melchionna, H.~Ranetbauer, L.~Scarpa, and L.~Trussardi,
\newblock From nonlocal to local {C}ahn-{H}illiard equation.
\newblock {\em Adv. Math. Sci. Appl.}, 28(2):197--211, 2019.

\bibitem{mirqui}
A.~Miranville and R.~Quintanilla, 
\newblock A generalization of the Caginalp phase-field system based on the Cattaneo law.
\newblock {\em Nonlinear Anal.}  71(5-6):2278--2290, 2009.

\bibitem{poiatti-SEP}
A.~Poiatti,
\newblock The 3d strict separation property for the nonlocal {C}ahn-{H}illiard
  equation with singular potential.
Preprint arXiv:2303.07745 [math.AP], 2023, pp.~1--27.

\bibitem{ponce}
A.~C. Ponce,
\newblock A new approach to {S}obolev spaces and connections to
  {$\Gamma$}-convergence.
\newblock {\em Calc. Var. Partial Differential Equations}, 19(3):229--255,
  2004.
  
\bibitem{rocschi} 	
E.~Rocca and G.~Schimperna, 
\newblock Global attractor for a parabolic-hyperbolic Penrose-Fife phase field system.
\newblock {\em Discrete Contin. Dyn. Syst.}  15(4):1193--1214, 2006.

\bibitem{sand-serf}
E.~Sandier and S.~Serfaty,
\newblock Gamma-convergence of gradient flows with applications to
  {G}inzburg-{L}andau.
\newblock {\em Comm. Pure Appl. Math.}, 57(12):1627--1672, 2004.

\bibitem{Sim87}
	 J.\ {S}imon, 
	\newblock Compact sets in the space $L^p(0,T;B)$, 
	\newblock {\em Ann.\ Mat.\ Pura.\ Appl.\ (4)}, 146:65--96, 1987.

\end{thebibliography}
\end{document}